\documentclass{article}
\usepackage{amsmath,amssymb,amsthm,amscd,dsfont}
\numberwithin{equation}{section} \allowdisplaybreaks
\begin{document}
\newtheorem{theorem}{Theorem}[section]
\newtheorem{defin}{Definition}[section]
\newtheorem{prop}{Proposition}[section]
\newtheorem{corol}{Corollary}[section]
\newtheorem{lemma}{Lemma}[section]
\newtheorem{rem}{Remark}[section]
\newtheorem{example}{Example}[section]
\title{Quasi-classical generalized CRF structures}
\author{{\small by}\vspace{2mm}\\Izu Vaisman}
\date{}
\maketitle
{\def\thefootnote{*}\footnotetext[1]%
{{\it 2010 Mathematics Subject Classification: 53C15, 53D17}.
}}
\begin{center} \begin{minipage}{12cm}
A{\footnotesize BSTRACT. In an earlier paper, we studied manifolds $M$ endowed with a generalized F structure $\Phi\in End(TM\oplus T^*M)$, skew-symmetric with respect to the pairing metric, such that $\Phi^3+\Phi=0$. Furthermore, if $\Phi$ is integrable (in some well-defined sense), $\Phi$ is a generalized CRF structure. In the present paper we study quasi-classical generalized F and CRF structures, which may be seen as a generalization of the holomorphic Poisson structures (it is well known that the latter may also be defined via generalized geometry). The structures that we study are equivalent to a pair of tensor fields $(A\in End(TM),\pi\in\wedge^2TM)$ where $A^3+A=0$ and some relations between $A$ and $\pi$ hold. We establish the integrability conditions in terms of $(A,\pi)$. They include the facts that $A$ is a classical CRF structure, $\pi$ is a Poisson bivector field and $im\,A$ is a (non)holonomic Poisson submanifold of $(M,\pi)$. We discuss the case where either $ker\,A$ or $im\,A$ is tangent to a foliation and, in particular, the case of almost contact manifolds. Finally, we show that the dual bundle of $im\,A$ inherits a Lie algebroid structure and we briefly discuss the Poisson cohomology of $\pi$, including an associated spectral sequence and a Dolbeault type grading.}
\end{minipage} \end{center} \vspace{5mm}
\section{Introduction}
In this paper, the differentiable manifolds and the differential geometric objects are $C^\infty$-smooth and the notation is that used in most textbooks on differential geometry.

Generalized geometry \cite{H} is the geometry of structures defined on the big tangent bundle $\mathbf{T}M=TM\oplus T^*M$ of a differentiable manifold $M^m$, endowed with the pairing metric
\begin{equation}\label{pairingg} g((X,\alpha),(Y,\beta))=\frac{1}{2}(\alpha(Y)+\beta(X))
\end{equation}
and the Courant bracket
\begin{equation}\label{Cbracket} [(X,\alpha),(Y,\beta)]=([X,Y],L_X\beta-L_Y\alpha+\frac{1}{2}d(\alpha(Y)-\beta(X))).
\end{equation}

In generalized geometry the most popular subject is generalized complex structures \cite{Galt}. A generalized almost complex structure is a complex, maximal $g$-isotropic subbundle	 $L\subseteq\mathbf{T}^cM=T^cM\oplus T^{*c}M$ such that $L\cap\bar{L}=0$\footnote{The upper index $c$ denotes complexification and the bar denotes complex conjugation.}. Equivalently, the structure is defined by the endomorphism $\mathcal{J}$ of $\mathbf{T}M$ with $\pm i$-eigenbundles $L,\bar{L}$, which is characterized by the properties (i) $\mathcal{J}$ is $g$-skew-symmetric, (ii) $\mathcal{J}^2=-Id.$ Furthermore, the structure is integrable, or generalized complex, if $L$ is closed under the Courant bracket.

The endomorphism $\mathcal{J}$ may be expressed by a matrix of classical tensors \cite{{Galt},{V-gcm}}, which we call the {\it tensor components} of $\mathcal{J}$,
\begin{equation}\label{matriceJ} \mathcal{J}\left(
\begin{array}{c}X\vspace{2mm}\\ \alpha \end{array}
\right) = \left(\begin{array}{cc} A&\sharp_\pi\vspace{2mm}\\
\flat_\sigma&-A^*\end{array}\right)
\left( \begin{array}{c}X\vspace{2mm}\\
\alpha \end{array}\right) \end{equation} where $X\in TM,\alpha\in
T^*M$, $A\in End(TM)$, $\sigma\in\Omega^2(M)$ ($\Omega^k(M)$ is the space of differential $k$-forms of $M$), $\pi\in\chi^2(M)$ ($\chi^k(M)$ is the space of $k$-vector fields, i.e., contravariant, skew-symmetric tensor fields of $M$), $\flat_\sigma X=i(X)\sigma$, $\sharp_\pi\alpha=i(\alpha)\pi$ and $*$ denotes transposition ($A^*\alpha=\alpha\circ A$). Then, $g$-skew-symmetry holds and $\mathcal{J}^2=-Id$ is equivalent to
\begin{equation}\label{A2} A^2+\sharp_\pi\circ\flat_\sigma=-Id,\;\sharp_\pi\circ A^*=A\circ\sharp_\pi,\;\flat_\sigma\circ A=A^*\circ\flat_\sigma.\end{equation}

The integrability of the structure is equivalent to the set of conditions \cite{{Cr},{V-gcm}}
\begin{equation}\label{Crainic} \begin{array}{l}[\pi,\pi]=0,\; R_{(\pi,A)}=0,\; \mathcal{N}_A(X,Y)=\sharp_\pi[i(X\wedge Y)d\sigma],\vspace*{2mm}\\ d\sigma_A(X,Y,Z)=\sum_{Cycl(X,Y,Z)} d\sigma(AX,Y,Z),\end{array}\end{equation}
where $[\,,\,]$ is the Schouten-Nijenhuis bracket \cite{Psgeom},
\begin{equation}\label{NijA} \mathcal{N}_A(X,Y)=[AX,AY]-A[X,AY]-A[AX,Y]+A^2[X,Y]\end{equation}
is the Nijenhuis tensor,
\begin{equation}\label{Schouten}R_{(\pi,A)}(X,\alpha)=\sharp_\pi[L_X(A^*\alpha) -L_{AX}\alpha]-(L_{\sharp_\pi\alpha}A)(X) \end{equation}
is the Schouten concomitant ($L$ denotes Lie derivative) and $\sigma_A(X,Y)=\sigma(AX,Y)$.

If $\mathcal{J}(TM)\subseteq TM$ and $\mathcal{J}(T^*M)\subseteq T^*M$, equivalently, $\sigma=0,\pi=0$, $\mathcal{J}$ reduces to a classical almost complex structure $A$ and the integrability conditions (\ref{Crainic}) reduce to the integrability condition $\mathcal{N}_A=0$.

We propose to call $\mathcal{J}$ a {\it quasi-classical} structure if $\mathcal{J}(TM)\subseteq TM$, equivalently, $\sigma=0$. It is known that in this case the integrability conditions (\ref{Crainic}) are equivalent to the fact that $A$ is a complex structure and $\pi$ is a holomorphic Poisson structure on $(M,A)$, e.g., \cite{Gengoux}. Indeed, the third condition (\ref{Crainic}) becomes $\mathcal{N}_A=0$, i.e., $A$ is complex and $M$ has the local, complex analytic coordinates $(z^i)$. Then, since $R$ is a tensor, it suffices to check $R=0$ for $X=\partial/\partial z^i,\alpha=dz^j$ and $X=\partial/\partial z^i,\alpha=d\bar{z}^j$. In the first case, $R=0$ becomes
$$A[\sharp_\pi(dz^j),\frac{\partial}{\partial z^i}]=i[\sharp_\pi(dz^j),\frac{\partial}{\partial z^i}],$$ which holds iff $\pi$ is a holomorphic Poisson structure, because $\pi$ has no component of complex type $(1,1)$ and $\sharp_\pi(dz^j)$ is an $i$-eigenvector of $A$. In the second case, $R=0$ is implied by $\pi^{ij}=\pi^{ij}(z^k)$, which, in turn, is implied by $[\pi,\pi]=0$.

In \cite{VCRF} we studied a more general type of generalized structure that corresponds to K. Yano's F structure \cite{Y}. An F structure is an endomorphism $A\in End(TM)$ such that $A^3+A=0$. Then, $A$ has the eigenvalues $\pm i,0$ with the corresponding eigenbundles
$H,\bar{H}\subseteq T^cM,\,Q=ker\,A\subseteq TM$ and with $P=im\,A\subseteq TM$ such that $P^c=H\oplus\bar{H}$ and $TM=P\oplus Q$. The projections that correspond to the decomposition $T^cM=H\oplus\bar{H}\oplus Q^c$ are given by
\begin{equation}\label{projPQ} pr_H=-\frac{1}{2}(A^2+iA),\, pr_{\bar{H}}=-\frac{1}{2}(A^2-iA),\, pr_Q=A^2+Id,\,pr_P=-A^2.\end{equation}
Furthermore, $H$ is an almost CR structure and, if it is closed under the Lie bracket, it is a CR structure \cite{DT}. The CR condition is equivalent to
\begin{equation}\label{integrcuA} \mathcal{S}_A(X,Y)=[AX,AY]+A[AX,A^2Y]+A[A^2X,AY]-[A^2X,A^2Y]=0.
\end{equation}
Indeed, (\ref{integrcuA}) holds on arguments that are eigenvectors iff $H$ is closed under the Lie bracket. If (\ref{integrcuA}) holds, we will say that $A$ is an {\it F structure of the CR type}.

A {\it generalized F structure} is defined by an endomorphism $\Phi$ of $\mathbf{T}M$, which is $g$-skew-symmetric and satisfies the condition $\Phi^3+\Phi=0$. Thus, the eigenvalues of $\Phi$ are $\pm i,0$. We may represent $\Phi$ by the right hand side of (\ref{matriceJ}) but, the conditions on the tensor fields $A,\pi,\sigma$ will be different \cite{VCRF}. Equivalently, the structure may be defined by the $\pm i$-eigenbundles $E,\bar{E}$ and the $0$-eigenbundle $S$, where $E$ is a complex $g$-isotropic (possibly not maximal) subbundle of $\mathbf{T}^cM$ such that $E\cap\bar{E}^{\perp_g}=0$ \cite{VCRF}. Notice that $S^c=E^{\perp_g}\cap\bar{E}^{\perp_g}$ and we have the decomposition $\mathbf{T}^cM=E\oplus\bar{E}\oplus S^c$. The projections of $\mathbf{T}^cM$ on $E,\bar{E},S^c$ are defined by formulas (\ref{projPQ}) where $A$ is replaced by $\Phi$.

The structure defined by $\Phi$ is said to be integrable and, then, it is called a {\it generalized CRF structure} if the subbundle $E$ is closed under Courant brackets. As in the classical case, by looking at arguments that belong to the various eigenbundles of $\Phi$, we see that the integrability condition is equivalent to $\mathcal{S}_\Phi((X,\alpha),(Y,\beta))=0$, where $\mathcal{S}_\Phi$ is defined by (\ref{integrcuA}) with $A$ replaced by $\Phi$ and Lie brackets replaced by Courant brackets \cite{VCRF}. $C^\infty(M)$-bilinearity of $\mathcal{S}_\Phi$ follows from $\Phi^3+\Phi=0$.

If a generalized F structure $\Phi$ preserves the tangent and cotangent bundle of $M$, we have $\pi=0,\sigma=0$ and $\Phi$ may be identified with a classical F structure $A$. Then, if $\Phi$ is integrable $A$ is called a {\it classical CRF structure}. The conditions that characterize a classical CRF structure are stronger than the demand that the $i$-eigenbundle $H$ of $A$ is CR, namely, these conditions are \cite{VCRF},
\begin{equation}\label{clasintegr1} [H,H]\subseteq H,\;\;[H,Q^c]\subseteq H\oplus Q^c,
\end{equation} equivalently,
\begin{equation}\label{clasintegr2} \begin{array}{ll} \mathcal{N}_A(X,Y)=pr_Q[X,Y],& \forall X,Y\in P, \vspace*{2mm}\\ \mathcal{N}_A(X,Y)=0,&\forall X\in P,\,Y\in Q.\end{array}\end{equation}

In \cite{VCRF} we also discussed generalized F structures $\Phi$ with classical square, i.e., such that $\Phi^2$ preserves the tangent and cotangent bundle of $M$. These are characterized by the conditions
\begin{equation}\label{classquare} A\circ\sharp_\pi=\sharp_\pi\circ A^*,\;\flat_\sigma\circ A= A^*\circ\flat_\sigma,\end{equation}
equivalently,
\begin{equation}\label{piA}
\pi_A(\alpha,\beta)=\pi(A^*\alpha,\beta),\;\sigma_A(X,Y) =\sigma(AX,Y)\end{equation}
define a bivector and a $2$-form, respectively.

The aim of this paper is to study the generalized CRF structures introduced by the following definition, which was suggested by the generalized complex structures equivalent to holomorphic Poisson structures.
\begin{defin}\label{defquasi} {\rm A generalized F structure $\Phi$ such that $\Phi(TM)\subseteq TM$ and $\Phi^2(T^*M)\subseteq T^*M$ will be called a {\it quasi-classical generalized F structure}. If integrable, the structure will be called a {\it quasi-classical generalized CRF structure}.}\end{defin}

In Section 2, we describe the matrix of tensor components of a quasi-classical generalized F structure and we get the integrability conditions of the structure. These include the fact that the endomorphism $A$ is a classical CRF structure and the bivector field $\pi$ is a Poisson structure. Some conditions that relate between $A$ and $\pi$ must also be added. In Section 3, we discuss the case where one of the subbundles $P,Q$ defined by the endomorphism $A$ is a foliation. In the last section we show that the Poisson structure $\pi$ induces a Lie algebroid structure on the dual bundle $P^*$. Then, we define a spectral sequence that converges to the Poisson cohomology of $\pi$ and we show that second term of this sequence has a Dolbeault type grading.
\section{Quasi-classical F and CRF structures}
We begin by characterizing the quasi-classical generalized F structures.
\begin{prop}\label{propalg} A quasi-classical generalized F structure $\Phi$ is equivalent with a pair $(A,\pi)$ where $A$ is an F structure and $\pi$ a bivector field that is $A$-compatible in the sense that $\pi_A$ is again a bivector and $im\,\sharp_\pi\subseteq im\,A=P$.
\end{prop}
\begin{proof}
Definition \ref{defquasi} implies that a quasi-classical generalized F structure $\Phi$ has a classical square and that the following conditions hold
\begin{equation}\label{compatApi} \sigma=0,\;A\circ\sharp_\pi=\sharp_\pi\circ A^*
\end{equation}
(see (\ref{classquare})). The second condition (\ref{compatApi}) is equivalent to the fact that $\pi_A$ defined by (\ref{piA}) is a bivector field.

Accordingly, the matrix representation of the structure has the form
\begin{equation}\label{matricePhi} \Phi\left(
\begin{array}{c}X\vspace{2mm}\\ \alpha \end{array}
\right) = \left(\begin{array}{cc} A&\sharp_\pi\vspace{2mm}\\
0&-A^*\end{array}\right)
\left( \begin{array}{c}X\vspace{2mm}\\
\alpha \end{array}\right),\end{equation}
equivalently, $\Phi(X,\alpha)=(AX+\sharp_\pi\alpha,-A^*\alpha)$. Furthermore, $\Phi^3+\Phi=0$ is equivalent to
\begin{equation}\label{propcar} A^3+A=0,\;A\circ(\sharp_\pi\circ A^*-A\circ\sharp_\pi)=\sharp_\pi\circ(A^{*2}+Id)\end{equation} (look at the cases $\alpha=0$ and $X=0$).

Thus, firstly, $A$ must be a classical F-structure and, secondly, since $im(A^{*2}+Id)=ann\,P$ and (\ref{compatApi}) holds, the last condition (\ref{propcar}) becomes $\sharp_\pi(ann\,P)=0$. Finally, the relation
\begin{equation}\label{aux1}
<\beta,\sharp_\pi\alpha>=-<\alpha,\sharp_\pi\beta>=0,\; \alpha\in T^*M,\beta\in ann\,P
\end{equation}
shows that $\sharp_\pi(ann\,P)=0\Leftrightarrow im(\sharp_\pi)\subseteq P$.
\end{proof}
\begin{rem}\label{obsSQ} {\rm The endomorphism $\Phi$ vanishes on $Q\oplus ann\,P$ ($ann\,P\approx Q^*$) and a dimension argument tells us that the $0$-eigenbundle of $\Phi$ is $S=Q\oplus ann\,P$.}\end{rem}
\begin{prop}\label{proppilocal}
Consider the quasi-classical generalized F structure defined by $(A,\pi)$ and let $H,\bar{H},Q$ be the $\pm i,0$-eigenbundles of $A$. Then, if $(h_i), (\bar{h}_i),(q_j)$ are local bases of $H,\bar{H},Q$, respectively, the local expression of $\pi$ must be of the form
\begin{equation}\label{pilocal} \pi=\frac{1}{2}(\pi^{ij}h_i\wedge h_j +\bar{\pi}^{ij}\bar{h}_i\wedge \bar{h}_j)\hspace*{2mm}(\pi^{ij}+\pi^{ji}=0).\end{equation}\end{prop}
\begin{proof}
The decomposition $T^cM=H\oplus\bar{H}\oplus Q^c$ dualizes to $T^{*c}M=H^*\oplus\bar{H}^*\oplus Q^{*c}$, where the terms are the $\pm i,0$-eigenbundles of $A^*$ and we may identify $$H^*=ann(\bar{H}\oplus Q^c),
\bar{H}^*=ann(H\oplus Q^c), Q^*=ann\,P.$$ Since $\sharp_\pi(ann\,P)=0$, the local expression of $\pi$ has no terms containing $q_i$. Furthermore, for $\nu\in H^*,\bar{H}^*$, $A\sharp_\pi\nu=\sharp_\pi(A^*\nu)=\pm i\sharp_\pi\nu$. Hence,
$\sharp_\pi:H^*\rightarrow H, \sharp_\pi:\bar{H}^*\rightarrow\bar{H}$ and $\pi$ can take non-zero values only if evaluated on two arguments that are both either in $H^*$ or in $\bar{H}^*$. Therefore, $\pi$ has no term $\pi^{ij}h_i\wedge\bar{h}_j$ and, since $\pi$ is real, we have the required conclusion. (In (\ref{pilocal}) and in the whole paper we use the Einstein summation convention.) \end{proof}
\begin{prop}\label{EptPhi} The $i$-eigenbundle $E$ of the quasi-classical generalized F structure $\Phi$ is given by the formula
\begin{equation}\label{exprE} E=\{(Z-\frac{1}{2}\sharp_\pi\xi,\xi)\,/\,
Z\in H,\xi\in\bar{H}^*= ann(H\oplus Q)\}.\end{equation}
\end{prop}
\begin{proof} Using (\ref{matricePhi}), we may write $E=E_P+E_Q$, where
$$\begin{array}{l}
E_P=\{(X,\alpha)\in\mathbf{T}^cM\,/\,X\in P^c,\,AX+\sharp_\pi\alpha=iX,\,-A^*\alpha=i\alpha\}\vspace*{2mm}\\ E_Q=\{(Y,\beta)\in\mathbf{T}^cM\,/\,Y\in Q^c, \sharp_\pi\beta=iY,\,-A^*\beta=i\beta\}.\end{array}$$
In $E_Q$, since $im\,\sharp_\pi\subseteq P$, we have $Y=0,\sharp_\pi\beta=0$, therefore, $E_Q\subseteq \bar{H}^*\subseteq E_P$ and $E=E_P$. Furthermore, (\ref{projPQ}) implies
$$E=E_P=\{(X,\alpha)\,/\,X\in P^c,\alpha\in\bar{H}^*= ann(H\oplus Q), \sharp_\pi\alpha=-2pr_{\bar{H}}X\}.$$
Formula (\ref{exprE}) is a reformulation of this result.
\end{proof}

The integrability conditions of a quasi-classical generalized F structure are given by the following theorem.
\begin{theorem}\label{thdeintegr} The quasi-classical generalized F structure $\Phi$ defined by the pair of tensor fields $(A,\pi)$ is a quasi-classical generalized CRF structure iff (1) the endomorphism $A$ is a classical CRF structure, (2) $\pi$ is a Poisson structure, (3) $R_{(\pi,A)}(X,\beta)=0$ for $X\in P,\beta\in ann\,Q$ ($P=im\,A,Q=ker\,A$).\end{theorem}
\begin{proof}
The tensor fields $A,\pi$ satisfy the conditions stated in Proposition \ref{propalg} and we refer to the eigenbundles of $A$ in the notation below.
If $\Phi$ is quasi-classical, the condition $\mathcal{S}_\Phi((X,0),(Y,0))=0$ becomes $\mathcal{S}_A(X,Y)=0$, hence, it is equivalent with the fact that $A$ is a structure of the CR type.

Obviously, it suffices to look at arguments of the form $(X,0)$ and $(0,\alpha)$ separately.
The arguments $(X,0),(Y,0)$ were already considered and we need to compute $\mathcal{S}_\Phi((0,\alpha),(0,\beta))$ and $\mathcal{S}_\Phi((X,0),(0,\beta))$.
Moreover, since $\mathcal{S}_\Phi$ vanishes if one of the arguments belongs to $S=Q\oplus(ann\,P)$ (see Remark \ref{obsSQ}), it suffices to establish the integrability conditions for $(X,0),\, X\in P$ and $(0,\beta),\, \beta\in ann\,Q$.

Computing the Courant brackets involved and using the following consequences of (\ref{compatApi})
$$\pi(\alpha\circ A)=\pi(\alpha,\beta\circ A),\,\pi(\alpha\circ A^2)=\pi(\alpha,\beta\circ A^2),$$ we get
\begin{equation}\label{alphabeta} \begin{array}{l} \mathcal{S}_\Phi((0,\alpha),(0,\beta))=([\sharp_\pi\alpha,\sharp_\pi\beta] +
\sharp_\pi[L_{\sharp_\pi\alpha}(\beta\circ A^2) -L_{\sharp_\pi\beta}(\alpha\circ A^2)\vspace*{2mm}\\ \hspace*{5mm}-d(\pi(\alpha\circ A^2,\beta))],
L_{\sharp_\pi\beta}(\alpha\circ A)-L_{\sharp_\pi\alpha}(\beta\circ A)
-d(\pi(\alpha\circ A,\beta))\vspace*{2mm}\\ \hspace*{5mm}+[L_{\sharp_\pi\alpha}(\beta\circ A^2) -L_{\sharp_\pi\beta}(\alpha\circ A^2)-d(\pi(\alpha\circ A^2,\beta))]\circ A)
\end{array}\end{equation}

For $\alpha,\beta\in ann\,Q$, we have $\alpha\circ A^2=-\alpha,\beta\circ A^2=-\beta$ and
the vanishing of the vector part of (\ref{alphabeta}) yields the Poisson condition
\begin{equation}\label{integr22}  [\sharp_\pi\alpha,\sharp_\pi\beta]=
\sharp_\pi\{\alpha,\beta\}_\pi\;(\alpha,\beta\in ann\,Q),\end{equation}
where
\begin{equation}\label{pi1form} \{\alpha,\beta\}_\pi=L_{\sharp_\pi\alpha}\beta- L_{\sharp_\pi\beta}\alpha-d(\pi(\alpha,\beta))\end{equation}
is the Poisson bracket of $1$-forms \cite{Psgeom}.

Since $\forall\alpha,\beta,\gamma\in T^*M$ the Gelfand-Dorfman formula \cite{GD}
\begin{equation}\label{GD}
[\pi,\pi](\alpha,\beta,\gamma)=2[\gamma(\sharp_\pi\{\alpha,\beta\}_\pi- [\sharp_\pi\alpha,\sharp_\pi\beta])
\end{equation} holds, condition (\ref{integr22}) is equivalent to
\begin{equation}\label{integr222} [\pi,\pi](\alpha,\beta,\gamma)=0,\; \forall\alpha,\beta\in ann\,Q,\gamma\in T^*M.\end{equation}

But, we can show that (\ref{integr222}) holds iff $\pi$ is a Poisson bivector, in other words,  $[\pi,\pi]=0$, equivalently (\ref{integr22}), hold for any arguments. Indeed, since $[\pi,\pi]$ is totally skew-symmetric, (\ref{integr222}) means that $[\pi,\pi]=0$ whenever at least two arguments belong to $ann\,Q$. On the other hand, if at least two arguments, e.g., $\alpha,\beta\in ann\,P$, then, $\sharp_\pi\alpha=0,\sharp_\pi\beta=0$ and $\{\alpha,\beta\}_\pi=0$ (see (\ref{pi1form})), therefore, $[\pi,\pi]=0$ because of (\ref{GD}). Since this covers all the possible cases in the decomposition $T^*M=(ann\,Q)\oplus(ann\,P)$, we are done.

Furthermore, the vanishing of the covector part of (\ref{alphabeta}) is the condition
\begin{equation}\label{integr3}	 \{\alpha,\beta\}_\pi\circ A= L_{\sharp_\pi\alpha}(\beta\circ A)-L_{\sharp_\pi\beta}(\alpha\circ A)-d(\pi(\alpha\circ A,\beta)),\;\alpha,\beta\in ann\,Q,\end{equation}
which is known for generalized complex structures (e.g., \cite{V-gcm}).
The insertion of the value of $\{\alpha,\beta\}_\pi$ given by (\ref{pi1form}) in (\ref{integr3}) shows that the latter is equivalent to the vanishing of the bivector field (see \cite{V-gcm})
\begin{equation}\label{Schbiv}C_{(\pi,A)}(\alpha,\beta)=\beta\circ L_{\sharp_\pi\alpha}A
-\alpha\circ L_{\sharp_\pi\beta}A +d(\pi(\alpha,\beta))\circ A-d(\pi(\alpha\circ A,\beta))=0. \end{equation}

Now, we shall discuss the condition $\mathcal{S}_\Phi((X,0),(0,\beta))=0$ for $X\in P,\beta\in ann\,Q$, i.e., $A^2X=-X,\beta\circ A^2=-\beta$ and
$$\begin{array}{l}
\Phi(X,0)=(AX,0),\,\Phi^2(X,0)=-(X,0),\vspace*{2mm}\\ \Phi(0,\beta)=(\sharp_\pi\beta,-\beta\circ A),\, \Phi^2(0,\beta)=-(0,\beta).\end{array}$$ The required value of $\mathcal{S}_\Phi$ is
$$\begin{array}{rcl}
\mathcal{S}_\Phi((X,0),(0,\beta))&=& ([AX,\sharp_\pi\beta]-A[X,\sharp_\pi\beta] +\sharp_\pi[L_X(\beta\circ A)-L_{AX}\beta],\vspace*{2mm}\\&-&L_{AX}(\beta\circ A)-L_X\beta -[L_X(\beta\circ A)-L_{AX}\beta]\circ A).
\end{array}$$

The vector component of the previous expression vanishes iff
\begin{equation}\label{integr4} (L_{\sharp_\pi\beta}A)(X)= \sharp_\pi[L_X(\beta\circ A)-L_{AX}\beta]\,\Leftrightarrow\,R_{(\pi,A)}(\beta,X)=0,\end{equation}
where $R_{(\pi,A)}$ is the {\it Schouten concomitant}. Due to the general relation \cite{VP-Nij} $$\alpha(R_{(\pi,A)}(\beta,X))=<C_{(\pi,A)}(\alpha,\beta),X>,$$ this condition is equivalent to (\ref{Schbiv}) and only one of them, e.g., (\ref{integr4}) must be required.

Finally, the covector component of the condition $\mathcal{S}_\Phi((X,0),(0,\beta))=0$ is
\begin{equation}\label{lastcond}
[L_{AX}\beta-L_X(\beta\circ A)]\circ A= L_{AX}(\beta\circ A)+L_X\beta,\;\;(X\in P, \beta\in ann\,Q)
\end{equation} and it splits into the cases (i) $AX=iX,\beta\circ A=i\beta$, (ii) $AX=iX,\beta\circ A=-i\beta$ and conjugates. In case (i) the condition holds trivially. In case (ii) the condition becomes $(L_X\beta)\circ A=-i(L_X\beta)$, equivalently,
\begin{equation}\label{eqlastcond} L_X\beta\in ann(H\oplus Q^c),\;\;\forall X\in H,\beta\in ann(H\oplus Q^c).
\end{equation}
Condition (\ref{eqlastcond}) means $<L_X\beta,Z>=0$ for $Z\in H\oplus Q^c$, which is equivalent to $[X,Z]\in H\oplus Q^c$ , $\forall X\in H, Z\in Q^c$. Together with the condition $[H,H]\subseteq H$, which holds because $\mathcal{S}_A=0$, we have (\ref{clasintegr1}), therefore, $A$ is a classical CRF structure.

The conclusion of the theorem is the sum of the conditions deduced during the proof.
\end{proof}

Notice that, even in the integrable case, condition (3) may not hold for other kind of arguments $(X,\beta)$.
\begin{rem}\label{obsclCRF} {\rm We know that (\ref{clasintegr1}) is equivalent to  (\ref{clasintegr2}). For any arguments, the Nijenhuis tensor of $A$ is given by
\begin{equation}\label{NijPP}\begin{array}{lcl}
\mathcal{N}_A(X,Y)&=&[AX,AY]-A[AX,Y]-A[X,AY]+A^2[X,Y]\vspace*{2mm}\\ &=&[AX,AY]-A[AX,Y]-A[X,AY]-[X,Y]\vspace*{2mm}\\ &+&pr_Q[X,Y].\end{array}\end{equation}
Therefore, the first condition (\ref{clasintegr2}) is equivalent to
$$[AX,AY]-A[AX,Y]-A[X,AY]-[X,Y]=0,\;\forall X,Y\in P,$$
which is equivalent to $[H,H]\subseteq H$ (check this on $\pm i$-eigenvectors $X,Y$). Then, if we take $X\in P,Y\in Q$, (\ref{NijPP}) shows that the second condition
(\ref{clasintegr2}) takes the form
\begin{equation}\label{eqobsclas} A^2[X,Y]-A[AX,Y]=0, \;\;\forall X\in P,Y\in Q.\end{equation}
Thus, $A$ is a classical CRF structure iff $[H,H]\subseteq H$ and (\ref{eqobsclas}) holds.}\end{rem}
\begin{prop}\label{echicuSch} In Theorem \ref{thdeintegr}, condition (3) may be replaced by the pair of conditions
\begin{equation}\label{bunCRFPs} [\sharp_\pi\beta,X]\in P,
\;\;(L_Y\pi)(\lambda,\mu)=0,\end{equation}
where $X\in P,Y\in\bar{H},\beta\in ann\,Q,\lambda,\mu\in ann(\bar{H}\oplus Q^c)$.
\end{prop}
\begin{proof} We will replace condition (3) of Theorem \ref{thdeintegr}, written under the form
(\ref{integr4}) by its evaluation on $1$-forms $\lambda$. For $\lambda\in ann\,P$,
(\ref{integr4}) becomes $(L_{\sharp_\pi\beta}A)(X)=0$, which is equivalent to
\begin{equation}\label{4ptannP} <\lambda,[\sharp_\pi\beta,AX]>=0,\;\forall X\in P,\,\lambda\in ann\,P.\end{equation}
This is equivalent to the first condition (\ref{bunCRFPs}) and	 it means that the Hamiltonian vector fields $\sharp_\pi\beta$ preserve the distribution $P$. The tensorial character of this condition follows from $X\in P,\beta\in ann\,Q$.

For $\lambda\in ann\,Q^c=H^*\oplus\bar{H}^*$ ($H^*=ann(\bar{H}\oplus Q^c)$, $\bar{H}^*=ann(H\oplus Q^c)$), since $X\in P^c=H\oplus\bar{H}$, the evaluation splits into the cases: (i) $AX=iX,\, A^*\beta=i\beta$, (ii) $AX=iX,\, A^*\beta=-i\beta$ and their complex conjugates and the results are
$${\rm(i)}\hspace{3mm} A[\sharp_\pi\beta,X]-i[\sharp_\pi\beta,X]=0,\hspace{3mm}
{\rm(ii)}\hspace{3mm} A[\sharp_\pi\beta,X]-i[\sharp_\pi\beta,X]=2i\sharp_\pi(L_X\beta).$$

Condition (i) holds because $\sharp_\pi\beta,X\in H$ and $H$ is closed by brackets since $A$ is of the CR type. Condition (ii) again splits into two cases (a) $\lambda\in H^*$, (b) $\lambda\in\bar{H}^*$. In case (a) we have $\lambda\circ A=i\lambda$ and the evaluation of $\lambda$ on the left hand side of (ii) gives zero. The same holds for the right hand side:
$$<\lambda,\sharp_\pi(L_X\beta)>=-<L_X\beta,\sharp_\pi\lambda>=-X(\pi(\lambda,\beta)) +<\beta,[X,\sharp_\pi\lambda]>=0,$$
because of (\ref{pilocal}) and of $[X,\sharp_\pi\lambda]\in H$.
But, in case (b) $\lambda\circ A=-i\lambda$ and the left hand side of (ii) becomes
$$-2i<\lambda,[\sharp_\pi\beta,X]>=2i(<\lambda,\sharp_{L_X\pi}\beta+ \sharp_\pi(L_X\beta)>.$$

Thus, (ii) becomes
\begin{equation}\label{integrtiphol} (L_X\pi)(\beta,\lambda)=0,\;\;\forall\beta,\lambda\in ann(H\oplus Q^c).\end{equation}
This condition is tensorial in $X$ since
$$(L_{fX}\pi)(\beta,\lambda)=f(L_{X}\pi)(\beta,\lambda) -\beta(X)\pi(df,\lambda)-\lambda(X)\pi(\beta,df)$$
and $X\in H$ implies $\beta(X)=0,\lambda(X)=0$.
By conjugation, (\ref{integrtiphol}) yields the second condition (\ref{bunCRFPs}).
\end{proof}
\begin{corol}\label{corolprodus} Let $(A_u,\pi_u)$ define quasi-classical generalized CRF structures on manifolds $M_u$, $u=1,2$. Then, $(A_1+A_2,\pi_1+\pi_2)$ defines a quasi-classical generalized CRF structure on $M_1\times M_2$.\end{corol}
\begin{proof} Checking the conditions of Proposition \ref{propalg} and Theorem \ref{thdeintegr} is straightforward.
\end{proof}
\begin{rem}\label{obsaux1} {\rm We shall indicate the following equivalent expressions of the first condition (\ref{bunCRFPs}). This condition is equivalent to $L_{\sharp_\pi\beta}A^2(X)=0$, $\forall X\in P,\,\forall\beta\in ann\,Q$. On the other hand, the form (\ref{4ptannP}) of the condition is equivalent to $d\lambda(\sharp_\pi\beta,X)=0$ and, then, to $L_{\sharp_\pi\beta}\lambda\in ann\,P$, $\forall\lambda\in ann\,P$.}\end{rem}
\section{Quasi-classical structures with foliations}
The simplest example of a quasi-classical generalized CRF structure is that of a locally product structure defined by two involutive subbundles $P,Q\subseteq TM$ (foliations), where the leaves of $P$ are endowed with a holomorphic Poisson structure $(A_P\in End\,P,\pi_P\in\wedge^2P)$. Then, we also have $\pi\in\chi^2(M)$ and, if we extend $A_P$ to $A\in End(TM)$ by $A|_Q=0$, it is easy to check all the conditions stated in Proposition \ref{propalg} and Theorem \ref{thdeintegr}, respectively, Proposition \ref{echicuSch}. (It suffices to use vector fields in $P$, respectively $Q$, with local components that depend only on the coordinates on the leaves of $P$, respectively $Q$.)

Below, we shall discuss quasi-classical generalized CRF structures where either $Q$ or $P$ is a foliation.
\begin{prop}\label{transvhol1} Let $\Phi$ be a quasi-classical generalized CRF structure defined by a pair $(A,\pi)$ such that $Q=ker\,A$ defines a foliation and the bivector field $\pi$ is projectable to the local transversal submanifolds of the leaves of $Q$. Then, $A$ induces a transversal holomorphic structure of $Q$ and $\pi$ projects to a $Q$-transversal holomorphic Poisson structure.\end{prop}
\begin{proof}
First, we will show that an F structure $A$ of the CR type with an integrable $0$-eigenbundle $Q$  is a classical CRF structure iff the tensor field $A$ is projectable to the local transversal submanifolds of the leaves of $Q$. Indeed, if $A$ is classical CRF, we have (\ref{eqobsclas}) (Remark \ref{obsclCRF}). If we assume there that $X\in P$ is a $Q$-projectable vector field, equivalently, $\forall Y\in Q$, $[X,Y]\in Q$, (\ref{eqobsclas}) becomes $A[AX,Y]=0$, whence $[AX,Y]\in Q$ and $AX$ is projectable too. This exactly is what projectability of $A$ means. Conversely, if $A$ and $X$ are projectable, $[X,Y],[AX,Y]\in Q$ and (\ref{eqobsclas}) holds for a projectable $X$. It is easy to check that this implies (\ref{eqobsclas}) for $X$ replaced by $fX$, where $f$ is an arbitrary function, i.e., (\ref{eqobsclas}) holds for any arguments. Notice that, if it exists, the projection $\tilde{A}$ of $A$ to the local $Q$-transversal submanifolds is given by $\tilde{A}[X]_Q=[AX]_Q$, where $X$ is projectable and the index Q denotes the image in the quotient space.

In the proposition, $\Phi$ is a quasi-classical generalized CRF structure, hence, $A$ is classical CRF and the induced tensor field $\tilde{A}$ exists. For any projectable $X$, we have $$\tilde{A}^2[X]_Q=[A^2X]_Q=[(A^2+Id)(X)-X]_Q=-[X]_Q,$$ which means that $\tilde{A}$ is almost complex. Let us compute the Nijenhuis tensor $\mathcal{N}_{\tilde{A}}$ by using foliated vector fields $X,Y\in P$. We get
$$\begin{array}{lll}
\mathcal{N}_{\tilde{A}}([X]_Q,[Y]_Q)&=&[\tilde{A}[X]_Q,\tilde{A}[Y]_Q] -\tilde{A}[\tilde{A}[X]_Q,[Y]_Q]\vspace{2mm}\\ &-&\tilde{A}[[X]_Q,\tilde{A}[Y]_Q]-[[X]_Q,[Y]_Q]\vspace*{2mm}\\ &=&[[AX,AY]-A[AX,Y]-A[X,AY]-[X,Y]]_Q\vspace*{2mm}\\ &=&[\mathcal{N}_A(X,Y)-pr_Q[X,Y]]_Q=0.\end{array}$$
The first equality holds because the Lie bracket is compatible with the projection to the $Q$-transversal submanifolds and the last equality holds because $H$ is involutive (see Remark \ref{obsclCRF}). This proves the existence of the transversal holomorphic structure of $Q$.

Now, we look at the Poisson bivector field $\pi$ of the CRF structure $\Phi$. Since $Q$ is a transversally holomorphic foliation, each point of $M$ has a coordinate neighborhood with local coordinates $(z^i,y^u)$, where $z^i$ are lifts of complex coordinates defined by $\tilde{A}$ on the local transversal submanifolds of $Q$ and $y^u$ are real coordinates on the leaves of $Q$ (e.g., see \cite{DK}). These produce local bases of $TM$ of the form
$$Z_i=\frac{\partial}{\partial z_i}-t_i^u\frac{\partial}{\partial y^u}\in H,\,
\bar{Z}_i=\frac{\partial}{\partial \bar{z}_i}-\bar{t}_i^u\frac{\partial}{\partial y^u}\in \bar{H},
Y_u=\frac{\partial}{\partial y^u},$$ where $t^u_i$ are some local functions of $(z,\bar{z},y)$. Thus, the expression (\ref{pilocal}) becomes
$$\pi=\frac{1}{2}(\pi^{ij}(z,\bar{z})Z_i\wedge Z_j+\bar{\pi}^{ij}(z,\bar{z})\bar{Z}_i\wedge \bar{Z}_j).$$
The local coefficients $\pi^{ij}$ do not depend on $y$ because of the hypothesis that $\pi$ is $Q$-projectable. (The projectability of $\pi$ means the existence of a bivector field $\tilde{\pi}$ on the local transversal submanifolds of the leaves of $Q$ such that $\pi$ and $\tilde{\pi}$ are related by the natural projection onto the submanifold, which happens iff $\pi^{ij}$ do not depend on $y$. The $Q$-projectability of $\pi$ is also equivalent to $im\,\sharp_\pi\subseteq P$ together with the fact that, locally, $\forall Y\in Q$, $L_Y\pi$ belongs to the ideal generated by the tangent vector fields of the leaves of $Q$.)

Furthermore, $\pi$ satisfies conditions (\ref{bunCRFPs}), in particular, $$(L_Y\pi)(\lambda,\mu)=0,\;\;\forall Y\in\bar{H}, \lambda,\mu\in ann(\bar{H}\oplus Q^c).$$ For $Y=\bar{Z}_i, \lambda,=dz^h, \mu=dz^k$, this yields $\bar{Z}_i(\pi^{hk})=0$.

Thus, the projection $\tilde{\pi}$ of $\pi$ is
$$\tilde{\pi}=\frac{1}{2}(\pi^{ij}(z)\frac{\partial}{\partial z_i}\wedge \frac{\partial}{\partial z_j}+\bar{\pi}^{ij}(\bar{z})\frac{\partial}{\partial \bar{z}_i}\wedge \frac{\partial}{\partial \bar{z}_j}).$$
Finally, the condition $[\tilde{\pi},\tilde{\pi}]=0$ holds because, like the Lie bracket, the Schouten-Nijenhuis bracket is compatible with the projection onto the local transversal submanifolds of the leaves of $Q$. Therefore, $\tilde{\pi}$ is a holomorphic Poisson structure on the local $Q$-transversal submanifolds, endowed with the complex structure $\tilde{A}$.\end{proof}
\begin{prop}\label{propPfol} Let $\Phi$ be a quasi-classical generalized F structure defined by the pair $(A,\pi)$ where the tangent subbundle $P=im\,A$ is a foliation. Then, $\Phi$ is integrable (i.e., CRF) iff (i) the pair $(A|_P,\pi|_{ann\,Q})$ ($Q=ker\,A$) defines holomorphic Poisson structures on the leaves of $P$, (ii) $\mathcal{N}_A(X,Y)=0$ $\forall X\in P,Y\in Q$.\end{prop}
\begin{proof}
Since $P=im\,A$ and $ann\,Q=P^*$, $(A|_P,\pi|_{ann\,Q})$ defines quasi-classical generalized, almost complex structures of the leaves of $P$ (see Introduction). If $\Phi$ is integrable, $A$ is a classical CRF structure and (\ref{clasintegr2}) implies (ii) of the proposition. On the other hand, the first condition (\ref{clasintegr2}) and formula (\ref{NijPP}) imply $\mathcal{N}_{A|_P}=0$. Therefore, $A$ induces complex analytic structures on the leaves of $P$ and $M$ is covered by local charts with real coordinates $(x^a)$ and complex coordinates $(z^i,\bar{z}^i)$ such that $P$ is defined by $dx^a=0$ and $(z^i,\bar{z}^i)$ are complex analytic coordinates along the leaves of $P$. Accordingly, the $0$-eigenbundle $Q^c$ of $A$ and $ann\,Q^c$ have local bases of the form $$X_a=\frac{\partial}{\partial x^a}-t_a^i\frac{\partial}{\partial z^i} -\bar{t}_a^i\frac{\partial}{\partial \bar{z}^i},\;\theta^i=dz^i+t_a^idx^a,\;\bar{\theta}^i= d\bar{z}^i+\bar{t}_a^idx^a.$$

Furthermore, the Poisson bivector field $\pi$, which, in view of (\ref{pilocal}),
has the local expression
\begin{equation}\label{piinPfol}
\pi=\frac{1}{2}(\pi^{il}(z,\bar{z},x)\frac{\partial}{\partial z_i}\wedge \frac{\partial}{\partial z_l}+\bar{\pi}^{il}(z,\bar{z},x)\frac{\partial}{\partial \bar{z}_i}\wedge \frac{\partial}{\partial \bar{z}_l}),\end{equation}
induces Poisson structures $\pi|_{ann\,Q}$ on the leaves of $P$ that have the same local expressions, but, with $x=const.$ Moreover, the second integrability condition (\ref{bunCRFPs}) for $\Phi$ reduces to
$$(L_{\frac{\partial}{\partial\bar{z}^i}}\pi)(\theta^h,\theta^k)= \frac{\partial\pi^{h,k}}{\partial\bar{z}^i}=0,$$ which means that $\pi$ is holomorphic along the leaves. Thus, integrability of $\Phi$ also implies (i).

Conversely, it is clear that condition (i) implies the existence of the local coordinates $(x^a,z^i,\bar{z}^i)$ described above and the local expression (\ref{piinPfol}), where $\partial\pi^{hk}/\partial\bar{z}^i=0$. This local expression implies $[\pi,\pi]=0$, hence, $\pi$ is a Poisson bivector field. The fact that $A$ is classical CRF is ensured by the integrability of $A|_P$ together with condition (ii), which are exactly the two conditions (\ref{clasintegr2}) in our case. Finally, the first condition (\ref{bunCRFPs}) is a consequence of the involutivity of $P$ and the second condition (\ref{bunCRFPs}) holds since we have $\partial\pi^{hk}/\partial\bar{z}^i=0$. Thus, all the conditions of Theorem \ref{thdeintegr} and Proposition \ref{echicuSch} hold and $\Phi$ is integrable.
\end{proof}
\begin{rem}\label{obssympl}{\rm If $\Phi$ is a quasi-classical generalized CRF structure such that the corresponding tensor fields satisfy the condition $im\,\sharp_\pi=im\,A$, then, $P=im\,A$ is the symplectic foliation of the Poisson structure $P$ and $(A|_P,\pi|_{ann\,Q})$ define holomorphic symplectic structures on the leaves of $P$.}\end{rem}

We exemplify Proposition \ref{propPfol} by the following structures.

A generalized almost contact structure of codimension $h$ is a system of tensor fields $(A\in End(TM),Z_a\in\chi(M),\pi\in\chi^2(M),\sigma\in\Omega^2(M),\xi^a\in\Omega^1(M))$ ($\chi(M)=\chi^1(M)$) that satisfies the following conditions \cite{Vstable}
\begin{equation}\label{condF}
\begin{array}{l}\pi(\alpha\circ A,\beta)=\pi(\alpha,\beta\circ A),\;
\sigma(AX,Y)=\sigma(X,AY),\vspace*{2mm}\\A(Z_a)=0,\;\xi^a\circ
A=0,\;i(Z_a)\sigma=0,\;i(\xi^a)\pi=0,\;\xi^a(Z_b)=\delta^a_b,\vspace*{2mm}\\ A^2=-Id-\sharp_\pi\circ\flat_\sigma+\sum_{a=1}^h\xi^a\otimes
Z_a.\end{array}\end{equation}
Furthermore, the structure is {\it normal} if \cite{Vstable}:
\begin{equation}\label{normalitate} \begin{array}{l}
[\pi,\pi]=0,\:R_{(\pi,A)}=0,\vspace*{2mm}\\ L_{Z_a}\pi=0,\:
L_{Z_a}\sigma=0,\:L_{\sharp_\pi\alpha}\xi^a=0,	
\vspace*{2mm}\\	
\mathcal{N}_A(X,Y)=\sharp_\pi(i(X\wedge Y)d\sigma) -
\sum_{a=1}^h(d\xi^a(X,Y))Z_a,\vspace*{2mm}\\
d\sigma_A(X,Y,Z)=\sum_{Cycl(X,Y,Z)}
d\sigma(AX,Y,Z)\vspace*{2mm}\\

[Z_a,Z_b]=0,\; L_{Z_b}\xi^a=0,\;L_{Z_a}A=0,\vspace*{2mm}\\	
(L_{AX}\xi^a)(Y) - (L_{AY}\xi^a)(X)=0.\end{array}\end{equation}

For a generalized almost contact structure of codimension $h$ such that $\sigma=0$, conditions (\ref{condF}) imply the fact that the pair $(A,\pi)$ defines a quasi-classical generalized F structure $\Phi$. Since $-1$ is not an eigenvalue of $\sharp_\pi\circ\flat_\sigma=0$, the normality of the structure $(A,\pi,0,Z_a,\xi^a)$ is characterized by the first four lines of (\ref{normalitate}) \cite{Vstable}. It is easy to see that these conditions imply the integrability of $\Phi$. In particular, if $X,Y\in P=ann\,\{\xi^a\}$ and since $Q=span\{Z_a\}$, we have
\begin{equation}\label{Nijauxcontact}
\mathcal{N}_A(X,Y)=-\sum_a d\xi^a(X,Y)Z_a=\sum_a \xi^a([X,Y])Z_a=pr_Q[X,Y].
\end{equation}
Furthermore, if $X\in P,Y\in Q$,
$$\begin{array}{lll}\mathcal{N}_A(X,Y)&=&-\sum_a d\xi^a(X,Y)Z_a=\sum_a(L_X\xi^a)(Y) \vspace*{2mm}\\ &=&\sum_a(L_{AX'}\xi^a)(Y)=\sum_a(L_{AY}\xi^a)(X)=0,\end{array}$$
where $X=AX'$ (we use $P=im\,A$) and the key fourth equality sign is given by the last line of (\ref{normalitate}) (of course, $AY=0$ because $Y\in Q$). These results about $\mathcal{N}_A$ show that $A$ is a classical CRF structure.
Condition $[Z_a,Z_b]=0$ included in (\ref{normalitate}) shows that the $0$-eigenbundle $Q=span\{Z_a\}$ is a foliation.

The geometric structure of a normal generalized contact structure of codimension $h$ was described in Theorem 3.3 of \cite{Vstable} and includes the $Q$-transversal holomorphic Poisson structure obtained in Proposition \ref{transvhol1}. In particular, Example 3.2 of \cite{Vstable} tells that, if $(N,A,\pi)$ is a holomorphic Poisson manifold and $M$ is a principal torus bundle over $N$ endowed with a connection $\xi$ of curvature form $\Xi$, then, the conditions (a) $i(\sharp_{\pi^h}\alpha)\Xi=0$, (b) $\Xi((A^hX,A^hY))=\Xi(X,Y)$ (where the upper index $h$ denotes the horizontal lift extended by zero on vertical arguments)	 ensure that $(M,A^h,\pi^h)$ is a quasi-classical, generalized CRF structure.

A classical almost contact structure \cite{Bl} is defined by a triple $(A\in End(TM),Z\in\chi(M),\xi\in\Omega^1(M))$ that satisfies (\ref{condF}) for $h=1,\pi=0,\sigma=0$, i.e.,
$$ A^2=-Id+\xi\otimes Z,\,AZ=0,\,A^*\xi=0,\,\xi(Z)=1.$$
Furthermore, the structure is normal if the normality conditions (\ref{normalitate}) hold, under the same restrictions. Then, we remain with the single condition
\begin{equation}\label{normal1} \mathcal{N}_A(X,Y)+d\xi(X,Y)Z=0,\;\forall X,Y\in\chi(M)
\end{equation}
and it implies
$$ L_ZA=0,\,L_Z\xi=0,\,(L_{AX}\xi)(Y)=(L_{AY}\xi)(X).$$

It follows easily that the almost contact structure $(A,Z,\xi)$ is normal iff (i) the F structure $A$ is of the CR type and (ii) $L_ZA=0$. Indeed, formula (\ref{Nijauxcontact}) with $h=1$ shows that the first condition (\ref{clasintegr2}), which is equivalent to $A$ being of the CR type (see Remark \ref{obsclCRF}), coincides with the normality condition (\ref{normal1}) for arguments $X,Y\in P=im\,A$ ($im\,A$ is defined by $\xi=0$). Therefore, normality implies (i) and (ii) and the latter imply the normality conditions for arguments in $P$. Then, for $X\in P$ and $Y=Z$, (\ref{normal1}) becomes $\mathcal{N}_A(X,Z)=\xi([X,Z])Z$ and it is implied by (i) and (ii) because $L_ZA=0$ implies the vanishing of both sides of the equality:
$$\begin{array}{l}
\mathcal{N}_A(X,Z)=-A[AX,Z]+A^2[X,Z]=A(L_ZA)(X)=0,\vspace*{2mm}\\
\xi([X,Z])=\xi([AU,Z])=-\xi((L_ZA)(U)+A[Z,U])=0\;\;\;(U\in P).\end{array}$$

Furthermore, if $(A,Z,\xi)$ is normal, $A$ is a classical CRF structure since it is of the CR type and $\mathcal{N}_A(X,Z)=0$ for $X\in P$.
\begin{defin}\label{contact-Ps} {\rm Let $(A,Z,\xi)$ be an almost contact structure on a manifold $M$. A Poisson bivector field $\pi\in\chi^2(M)$ is a {\it contact-Poisson structure} on $M$ if $(A,\pi)$ is a quasi-classical, generalized CRF structure.}\end{defin}

The integrability conditions of the quasi-classical structure defined by $(A,\pi)$ show that the bivector field $\pi$ is a contact-Poisson structure on $(M,A,Z,\xi)$ iff: (1) $\pi(\alpha\circ A,\beta)=\pi(\alpha,\beta\circ A),\, i(\xi)\pi=0$, (2) $A$ is a classical CRF structure, (3) $\pi$ is a Poisson bivector field, (4) $R_{(\pi,A)}(X,\beta)=0$ whenever $\xi(X)=0,\beta(Z)=0$. Condition (4) may be replaced by the conditions (\ref{bunCRFPs}). In particular, (4) implies $\xi([\sharp_\pi\alpha,X])=0$ whenever $\xi(X)=0$; for $\alpha(Z)=0$ this is the first condition (\ref{bunCRFPs}) and for $\alpha=\xi$ we have $\sharp_\pi\alpha=0$.

Of course, $Q=ker\,A=span\{Z\}$ is a foliation. If we also assume that $L_Z\pi=0$, it easily follows that $\pi$ is $Q$-projectable and we may apply Proposition \ref{transvhol1} and get the $Q$ transversal holomorphic Poisson structure defined by the projections of $A$ and $\pi$. This situation occurs for contact-Poisson structures on normal, almost contact manifolds.
\begin{defin}\label{normalPs} {\rm Let $(A,Z,\xi)$ be a normal almost contact structure on a manifold $M$ and $\pi$ a contact-Poisson structure on $M$. If $L_Z\pi=0$, $\pi$ will be called a {\it normal contact-Poisson structure}.}\end{defin}
\begin{example}\label{exnormPs} {\rm Let $(M,A,Z,\xi)$ be a normal almost contact manifold such that $\xi\wedge d\xi=0$. Then, $P=im\,A$ is integrable and, since $A$ is of the CR type, $A|_P$ defines	 complex structures on the leaves of $P$. Any $Z$-projectable bivector field $\pi$ that produces holomorphic Poisson structures of the leaves of $P$ obviously is a normal contact-Poisson structure on $M$. In particular, this is the case for a cosymplectic manifold in the sense of Blair \cite{Bl} because such a manifold satisfies the condition $d\xi=0$.
The structure $\pi=0$ is a trivial example of a normal contact-Poisson structure on any almost contact manifold $M$, where the subbundle $P$ may not be involutive. This trivial example can be modified as follows. Take $M=N\times N'$ where $N$ is a complex manifold with the complex structure tensor $J$ and the holomorphic Poisson structure $\pi\neq0$ and $N'$ is a normal almost contact manifold with the structure $(A,Z,\xi)$, where $im\,A$ may not be involutive. Then, $(\tilde{A}=J+A,Z,\xi)$ is still a normal almost contact manifold and $im\,\tilde{A}$ may not be involutive. It is easy to check that $\pi$ is a normal contact-Poisson structure on $M$.}\end{example}
\begin{prop}\label{propcaznormal} If $(M,A,Z,\xi)$ is a normal almost contact manifold and $\pi$ is a normal contact-Poisson structure on $M$, then, the generalized almost contact structure $(A,\pi,Z,\xi)$ is normal.\end{prop}
\begin{proof} Under the hypotheses, we already have conditions (\ref{normalitate}) except for \begin{equation}\label{aux10}L_{\sharp_\pi\alpha}\xi=0,\;R_{(\pi,A)}(X,\beta)=0, \,\forall (X,\alpha).\end{equation}
The first condition (\ref{aux10}) follows since, for $\xi(X)=0$, an already mentioned consequence of property (4) of contact-Poisson structures gives
\begin{equation}\label{aux30}
<L_{\sharp_\pi\alpha}\xi,X>=-<\xi,[{\sharp_\pi\alpha},X]>=0\end{equation}
and
$$<L_{\sharp_\pi\alpha}\xi,Z>=-<\xi,[\sharp_\pi\alpha,Z]>= <\xi,\sharp_{L_Z\pi}\alpha+\sharp_\pi(L_Z\alpha)>=-<L_Z\alpha,\sharp_\pi\xi>=0.$$

Furthermore, the second condition (\ref{aux10}) holds under the restrictions of (4) and we have to check it for the arguments that do not satisfy these restrictions: $(A^2X,\xi),(Z,A^{*2}\alpha),(Z,\xi)$. We have
$$R_{(\pi,A)}(A^2X,\xi)=\sharp_\pi[L_{A^2X}(A^*\xi)-L_{A^3X}\xi]-(L_{\sharp_\pi\xi} A)(A^2X)
=\sharp_\pi(L_{AX}\xi)$$
and, also, for all $\alpha$,
$$<\alpha,\sharp_\pi(L_{AX}\xi)>=-<L_{AX}\xi,\sharp_\pi\alpha>=<\xi,[\sharp_\pi\alpha,AX]>=0.$$

Then, $$\begin{array}{lcl}
R_{(\pi,A)}(Z,A^{*2}\alpha)&=&\sharp_\pi[L_Z(A^{*3}\alpha)-L_{AZ}\alpha] -(L_{\sharp_\pi(A^{*2}\alpha)}A)(Z)\vspace*{2mm}\\	& =&-\sharp_\pi(A^*L_Z\alpha)-[\sharp_\pi(A^{*2}\alpha),AZ]+A[\sharp_\pi(A^{*2}\alpha),Z]\vspace*{2mm}\\	 &=&-A\sharp_\pi(L_Z\alpha)-A(\sharp_{L_Z\pi}(A^{*2}\alpha)-A^3\sharp_\pi(L_Z\alpha)=0.\end{array}$$

Finally, $R_{(\pi,A)}(Z,\xi)=0$ follows straightforwardly from the definition of $R_{(\pi,A)}$.
\end{proof}

If $(M_u,A_u,Z_u,\xi_u)$, $u=1,2$, are two almost contact manifolds, the formula
$$J(X_1,X_2)=(A_1X_1-\xi_2(X_2)Z_1,A_2X_2+\xi_1(X_1)Z_2),$$ where $X_u\in T_{x_u}M_u$, $x_u\in M_u$, $u=1,2$, defines an almost complex structure on $M_1\times M_2$ and it was proven in \cite{Mor} that $J$ is integrable iff the two almost contact structures are normal. In this case, the following proposition is true.
\begin{prop}\label{produshol} Let $(M_u,A_u,Z_u,\xi_u)$, $u=1,2$, be two normal almost contact manifolds	 that have normal contact-Poisson structures $\pi_u$. Then, $\pi=\pi_1+\pi_2$ is a holomorphic Poisson structure on the complex analytic manifold $(M_1\times M_2,J)$.\end{prop}
\begin{proof} Obviously, $\pi$ is a Poisson structure. Hence, conditions (\ref{Crainic}) reduce to
\begin{equation}\label{aux20}
R_{(\pi,J)}((X_1,X_2),(\alpha_1,\alpha_2))=0,\end{equation}
where, because $R$ is a tensor on $M_1\times M_2$, we may assume that $X_u\in\chi(M_u),\alpha_u\in\Omega^1(M_u)$. We need to check (\ref{aux20}) for the following type of arguments: (1) $(X_1,0),(\alpha_1,0)$, (2) $(X_1,0),(0,\alpha_2)$, (3) $(0,X_2),(\alpha_1,0)$ (4) $(0,X_2),(0,\alpha_2)$.

Firstly, we notice the expression of the transposed operator $J^*$:
$$J^*(\alpha_1,\alpha_2)=(A_1^*\alpha_1+\alpha_2(Z_2)\xi_1, A_2^*\alpha_2-\alpha_1(Z_1)\xi_2).$$
Now, using the definition (\ref{Schouten}) of the Schouten concomitant, a straightforward calculation gives
$$R_{(\pi,J)}((X_1,0),(\alpha_1,0))=(R_{(\pi_1,A_1)}(X_1,\alpha_1), \xi_1([\sharp_{\pi_1}\alpha_1,X_1])Z_2),$$
which vanishes because the normality of the first manifold implies $R_{(\pi_1,A_1)}(X_1,\alpha_1)=0$ and (\ref{aux30}) for the first manifold is $\xi_1([\sharp_{\pi_1}\alpha_1,X_1])=0$. In the same way, we will obtain (\ref{aux20}) in case (4).

Then, starting with (\ref{Schouten}) we get
$$\begin{array}{lcl}
R_{(\pi,J)}(X_1,0),(0,\alpha_2)&=& (\alpha_2(Z_2)\sharp_{\pi_1}(L_{X_1}\xi_1),0)
-(0,\xi_1(X_1)\sharp_{\pi_2}(L_{Z_2}\alpha_2))\vspace*{2mm}\\ &&
-(0,\xi_1(X_1)[\sharp_{\pi_2}\alpha_2,Z_2]).
\end{array}$$
The property $L_{Z_2}\pi_2=0$ implies the cancelation of the last two terms and $\sharp_{\pi_1}(L_{X_1}\xi_1)=0$ follows from the evaluation on any $1$-form $\beta_1$:
$$<\sharp_{\pi_1}(L_{X_1}\xi_1),\beta_1>=-<L_{X_1}\xi_1,\sharp_{\pi_1}\beta_1>=
<\xi_1,[X_1,\sharp_{pi_1}\beta_1]>\stackrel{(\ref{aux30})}{=}0.$$
Thus, (\ref{aux20}) holds in case (2). In case (3), (\ref{aux20}) follows similarly.
\end{proof}
\section{Poisson cohomology}
In this section we make some remarks on the Poisson cohomology defined by the Poisson structure $\pi$ of a quasi-classical generalized CRF manifold $(M,A,\pi)$. This is a particular case of a more general situation worthy of consideration.

In an older terminology, a subbundle $P\subseteq TM$ is a (non)holonomic submanifold of $M$ (``non" is omitted in the involutive case).
\begin{defin}\label{nonholPssub} {\rm A (non)holonomic submanifold $P$ of the Poisson manifold $(M,\pi)$ is a {\it (non)holonomic Poisson submanifold} if: 1) $im(\sharp_\pi:T^*M\rightarrow TM)\subseteq P$, 2) every infinitesimal transformation $\sharp_\pi\alpha$ ($\alpha\in\Omega^1(M)$) preserves the distribution $P$.}\end{defin}

Condition 2) means that
\begin{equation}\label{Hamilt-P} [\sharp_\pi\alpha,X]\in P,\;\forall\alpha\in\Omega^1(M),\, X\in P.\end{equation}
If condition 1) holds, condition 2) is coherent since (\ref{Hamilt-P}) is consistent with the multiplication of either $\alpha$ or $X$ by a function.

If $\Phi$ is a quasi-classical generalized CRF structure defined by the pair $(A,\pi)$, Proposition \ref{propalg} and Proposition \ref{echicuSch} show that $P=im\,A$ is a (non)holonomic Poisson submanifold of $(M,\pi)$.

If $\pi$ of Definition \ref{nonholPssub} is non-degenerate, we have $P=TM$ and (\ref{Hamilt-P}) automatically holds. If $\pi$ is a regular Poisson structure and $P=im\,\sharp_\pi$, $P$ is the symplectic foliation of $\pi$ (hence, holonomic) and (\ref{Hamilt-P}) again holds because $X$ is tangent to the symplectic leaves of $\pi$.
In the general holonomic case, $P$ is integrable and condition 1) of Definition \ref{nonholPssub} means that every leaf of $P$ is a Poisson submanifold, i.e., a union of open subsets of symplectic leaves of $\pi$ \cite{CZ}. Condition 2) is again implied because the brackets compute along the leaves of $P$. Thus, the notion of a (non)holonomic Poisson submanifold is a natural extension of the notion of a Poisson submanifold.

An example of a nonholonomic Poisson submanifold is given below.
\begin{example}\label{exnonholsb} {\rm Take
$M=\mathds{R}^5$ with the cartesian coordinates $(y^1,y^2,x^1,x^2,x^3)$ and the Poisson bivector field $$\pi=f(y^1,y^2,x^1,x^2,x^3)\frac{\partial}{\partial y^1}\wedge\frac{\partial}{\partial y^2}.$$ Consider the distribution $$P=span\{\frac{\partial}{\partial y^1}, \frac{\partial}{\partial y^2}, X_1=\frac{\partial}{\partial x^1} +x^2\frac{\partial}{\partial x^3}, X_2=\frac{\partial}{\partial x^2}-x^1\frac{\partial}{\partial x^3}\}.$$ Then, $P$ includes $im\,\sharp_\pi$, which is zero where $f$ vanishes and is spanned by $\partial/\partial y^1,\partial/\partial y^2$ where $f\neq0$. $P$ satisfies condition 2) of Definition \ref{nonholPssub} since $[\partial/\partial y^a,X_u]=0$ for $a=1,2,u=1,2$ and
$P$ is nonholonomic since $[X_1,X_2]=-2(\partial/\partial x^3)$ is not contained in $P$.}\end{example}

The Poisson cohomology of a Poisson manifold $(M,\pi)$ is the de Rham cohomology of the Lie algebroid $(T^*M,\sharp_\pi,\{\,,\,\}_\pi)$, i.e., the cohomology of the cochain complex $(\chi^k(M),d_\pi)$, where the coboundary $d_\pi$ is given by \cite{Psgeom}
\begin{equation}\label{coboundary}\begin{array}{r}
d_{\pi}w(\lambda_0,...,\lambda_k)= \sum_{h=0}^k(-1)^h(\sharp_\pi\lambda^h)(w(\lambda_0,...,
\hat{\lambda}_h,...,\lambda^k))\vspace*{2mm}\\	+ \sum_{h<s}(-1)^{h+s}w(\{\lambda_h,\lambda_s\}_\pi,\lambda_0,..., \hat{\lambda}_h,...,\hat{\lambda}_s,...,\lambda_k),
\end{array}\end{equation}
the hat denoting a missing argument. Another expression of the coboundary is $d_{\pi}w=-[\pi,w]$ (Schouten-Nijenhuis bracket). In fact, every Lie algebroid $\mathcal{A}$ has de Rham cohomology spaces $H^k(\mathcal{A})$ defined by the cochain complex $(\Gamma\wedge^k\mathcal{A}^*,d_{\mathcal{A}})$, where $d_{\mathcal{A}}$ has the same expression as $d_\pi$, but, $\sharp_\pi$ is replaced by the anchor of $\mathcal{A}$ and the $\pi$-bracket is replaced by the bracket of $\mathcal{A}$.
The Poisson cohomology spaces are $H^k(M,\pi)=H^k(T^*M)$.

Concerning the Lie algebroid $T^*M$, we	 notice the following result.
\begin{prop}\label{propLiealg}
If $P$ is a (non)holonomic Poisson submanifold of the Poisson manifold $(M,\pi)$, the Lie algebroid $(T^*M,\sharp_\pi,\{\,,\,\}_\pi)$ induces a Lie algebroid structure on $P^*=T^*M/(ann\,\pi)$.
Furthermore, if we put anchor and bracket zero on $ann\,P$, the sequence
\begin{equation}\label{exactseq} 0\rightarrow ann\,P\stackrel{\iota}{\rightarrow} T^*M \stackrel{p}{\rightarrow}T^*M/(ann\,P)\rightarrow0,\end{equation}
where $\iota$ is the inclusion and $p$ is the natural projection, is an exact sequence of Lie algebroids over $M$.\end{prop}
\begin{proof}
Equality (\ref{aux1}) and its consequence that $im(\sharp_\pi:T^*M\rightarrow TM)\subseteq P$ is equivalent to $\sharp_\pi(ann\,P)=0$ hold for any pair $(\pi\in\chi^2(M), P\subseteq TM)$. Hence, if $im\,\sharp_\pi\subseteq P$, we get an induced morphism $\sharp'_\pi:T^*M/(ann\,P)\approx P^*\rightarrow P$.

Then, since $\sharp_\pi(ann\,P)=0$, formula (\ref{pi1form}) shows that, $\forall\alpha,\beta\in ann\,P$ one has $\{\alpha,\beta\}_\pi=0$. Moreover, if only one of the forms, say $\beta$, belongs to $ann\,P$, then, $\forall\gamma\in\Omega ^1(M)$,
$$\begin{array}{l} <\gamma,\sharp_\pi\{\alpha,\beta\}_\pi>=
<\gamma,\sharp_\pi(L_{\sharp_\pi\alpha}\beta>= -<L_{\sharp_\pi\alpha}\beta,\sharp_\pi\gamma>\vspace*{2mm}\\
=(\sharp_\pi\alpha)(\pi(\beta,\gamma)+<\beta,[\sharp_\pi\alpha,\sharp_\pi\gamma]> =<\beta,\sharp_\pi\{\alpha,\gamma\}_\pi>=0, \end{array}$$
where the penultimate equality sign is justified by the fact that $\pi$ is Poisson and the ultimate is justified by $\sharp_\pi(ann\,P)=0$. Thus, $\{\alpha,\beta\}_\pi\in ann\sharp_\pi$.

Accordingly, the formula
$$\{[\alpha],[\beta]\}_\pi=[\{\alpha,\beta\}_\pi],$$
where brackets denote classes in $T^*M/(ann\,P)$, yields a well defined bracket on $\Gamma P^*$. Obviously, if the Jacobi identity holds for the bracket of $1$-forms, which is true if $\pi$ is Poisson, it also holds for the bracket of classes.

The above proves the existence of the induced Lie algebroid structure of $P^*$. The exactness of the sequence (\ref{exactseq}) is obvious.
\end{proof}
\begin{rem}\label{obdDirac} {\rm Proposition \ref{propLiealg} may be generalized as follows. Let $L\subseteq\mathbf{T}M$ be a Dirac structure (i.e., a maximally $g$-isotropic subbundle closed under Courant brackets) on a manifold $M$ and $S=pr_{TM}L$. For any (non)holonomic submanifold $P\subseteq TM$, we denote \cite{V-Dircoupl}:
$$\begin{array}{l}
\tilde{H}(L,P)=\{(X,\alpha)\in L,\,/\,\alpha\in ann\,P\}=L\cap(TM\oplus ann\,P) \subseteq\mathbf{T}M,\vspace*{2mm}\\ H(L,P)=\{X\in TM,\;/\exists\alpha\in ann\,P,\;(X,\alpha)\in L\}=pr_{TM}\tilde{H}(L,P). \end{array}$$
Then, we will say that $P$ is a {\it(non)holonomic Dirac submanifold} of $(M,L)$ if (1) $S\subseteq P$, (2) for any vector fields $Z\in S,X\in P$, $[Z,X]\in P$.
Then, if $L\cap(ann\,P)$ is a differentiable vector bundle ($ann\,P$ is identified with a subspace of $TM\oplus T^*M$ by $\alpha\mapsto(0,\alpha)$ ($\alpha\in ann\,P$)), the sequence
\begin{equation}\label{exactseqD} 0\rightarrow L\cap(ann\,P) \stackrel{\iota}{\rightarrow} L \stackrel{p}{\rightarrow}L/L\cap(ann\,P)\rightarrow0,\end{equation}
with the brackets induced by the Courant bracket is an exact sequence of Lie algebroids. By (\ref{Cbracket}), the restriction of the Courant bracket to $T^*M$, and in particular to pairs $(0,\alpha),(0,\beta)\in L\cap ann\,P$, is zero. Also, for $\alpha\in ann\,P$ and $(Y,\beta)\in L$, we get
$$[(0,\alpha),(Y,\beta)]=(0,-L_Y\alpha)\in L\cap ann\,P$$
because, if $Z\in P$, $$L_Y\alpha(Z)=-<\alpha,[Y,Z]>\stackrel{(2)}{=}0.$$
Thus, the Courant bracket induces a Lie algebroid bracket on the quotient that appears in (\ref{exactseqD}).
The case of a Poisson manifold $(M,\pi)$ is $L=graph\,\sharp_\pi$.}\end{rem}

In the proof of Proposition \ref{propLiealg} we saw that, if $P$ is a (non)holonomic Poisson submanifold of $(M,\pi)$, $ann\,P$ is an abelian ideal of $(\Omega^1(M),\{\,,\,\}_\pi)$. This property was encountered in the case $P=im\,\sharp_\pi$, where $\pi$ is a regular Poisson structure, and it was used to define a spectral sequence that converges to the Poisson cohomology of $\pi$ \cite{Psgeom}. Below, we recall the definition of this spectral sequence, referring to an arbitrary (non)holonomic Poisson submanifold. As a matter of fact, the construction of the spectral sequence involves a complementary subbundle $Q$ of $P$, which yields the decomposition $TM=Q\oplus P$. In the case of a quasi-classical generalized CRF structure with associated tensors $(A,\pi)$, one has a canonical spectral sequence defined by $P=im\,A,Q=ker\,A$.
The decomposition $TM=P\oplus Q$ implies $$T^*M=Q^*\oplus P^*=(ann\,P)\oplus(ann\,Q).$$ Accordingly, we may transfer the Lie algebroid structure of $P^*$ given by Proposition  \ref{propLiealg} to $ann\,Q$ and replace $L/L\cap(ann\,P)$ by $ann\,Q$ in (\ref{exactseqD}).

The decomposition of $T^*M$ produces a bi-grading $\chi^k(M)=\sum_{i+j=k}\chi^{ij}(M)$, where $i,j\geq0$ and $i$ is the $Q$-degree, $j$ is the $P$-degree (i.e., $w\in\chi^{ij}(M)$ vanishes unless evaluated on $i$ $1$-forms in $Q^*=ann\,P$ and $j$ $1$-forms in $P^*=ann\,Q$).
Then, since $\sharp_\pi(ann\,P)=0$ and $ann\,P$ is an abelian ideal of $(\Omega^1(M),\{\,,\,\}_\pi)$, if we count the arguments in $d_\pi w$ given by (\ref{coboundary}) for $w\in\chi^{ij}(M)$, we see that $d_\pi=\sigma'_{-1,2}+\sigma''_{0,1}$, where the lower indices indicate the grade increments. Furthermore, the coboundary property $d_\pi^2=0$ is equivalent to
\begin{equation}\label{propsigma}\sigma'^2=0,\,\sigma''^2=0,\, \sigma'\circ\sigma''+\sigma''\circ\sigma'=0.\end{equation}

Formula (\ref{coboundary} yields the following expression of $\sigma',\sigma''$ for $w\in\chi^{ij}(M)$ \cite{Psgeom}
\begin{equation}\label{exprsigma} \begin{array}{l}
(\sigma'w)(\alpha_0,...,\alpha_{i-2},\beta_0,...,\beta_{j+1})\vspace*{2mm}\\
=\sum_{h<k=0}^{j+1}(-1)^{h+k}w(\{\beta_h,\beta_k\}_\pi,\alpha_0,...,\alpha_{i-2}, \beta_0,...,\hat{\beta}_h,...,\hat{\beta}_k,...,\beta_{j+1}),\vspace*{2mm}\\
(\sigma''w)(\alpha_0,...,\alpha_{i-1},\beta_0,...,\beta_j)\vspace*{2mm}\\
=\sum_{h=0}^j(-1)^{i+h}(\sharp_\pi\beta_h)(w(\alpha_0,..., \alpha_{i-1},\beta_0,...,\hat{\beta}_h,...,\beta_j)\vspace*{2mm}\\
+\sum_{h=0}^{i-1}\sum_{k=0}^j(-1)^{i+h+k}w(\{\alpha_h,\beta_k\}_\pi, \alpha_0,...,...,\hat{\alpha}_h,...,\alpha_{i-1},\beta_0,...,\hat{\beta}_k,..., \beta_j)\vspace*{2mm}\\ + \sum_{h<k=0}^j(-1)^{h+k}w(\alpha_0,...,\alpha_{i-1},\{\beta_h,\beta_k\}_\pi,
\beta_0,...,\hat{\beta}_h,...,\hat{\beta}_k,...,\beta_j),\end{array}\end{equation}
where the arguments $\alpha\in ann\,P$ and $\beta\in ann\,Q$.
\begin{rem}\label{obscohpart} {\rm The restriction of $\sigma''$ to $\sum_j\chi^{0j}(M)$ defines the cohomology of the Lie algebroid $P^*$. On the other hand, $H^j(ann\,P)=\Gamma\wedge^jQ$ since $ann\,P$ has zero anchor and bracket.}\end{rem}

Now, we define a filtration of the cochain complex $(\chi^\bullet(M),d_\pi)$ by the spaces of filtration degree $h$:
$$F_h(M)=(\oplus_{i=0}^q\chi^{i,h})\oplus(\oplus_{i=0}^q\chi^{i,h+1}) \oplus\,...\,\oplus(\oplus_{i=0}^q\chi^{ip}),$$
where $q=rank\,Q,p=rank\,P$. Obviously, $F_h(M)\supseteq F_{h+1}(M)$ and the type increments of $\sigma',\sigma''$ show that $d_\pi(F_h)\subseteq F_h$.

The filtration $F_h(M)$ produces a spectral sequence $\{E_r^{ij}(M,\pi)\}$, which converges to the Poisson cohomology of $(M,\pi)$. The first terms of the spectral sequence are given by Theorem 5.12 of \cite{Psgeom} since the proof of that theorem holds in the present case too:
$$\begin{array}{l} E_0^{ij}(M,\pi)=E_1^{ij}(M,\pi)=\chi^{ji}(M),\vspace*{2mm}\\ E_2^{ij}(M,\pi)=H^i(\chi^{j\bullet}(M),\sigma'').\end{array}$$
Notice that $\chi^{jk}(M)=\Omega^k(P^*,\wedge^jQ)$, the space of the $k$-forms on $P^*$ with values in $\wedge^jQ$.

Furthermore, (\ref{propsigma}) show that $\sigma'$ induces an action on the cohomology spaces that give $E_2^{ij}$ and the definitions related with spectral sequences yield
$$ E_3^{ij}(M,\pi)=H^i(H^\bullet(\chi^{j\bullet}(M),\sigma''),\sigma').$$

In the case of a quasi-classical generalized CRF manifold $(M,A,\pi)$ there exists a further grade decomposition induced by $P=H\oplus\bar{H}$. In order to express it in a simple way, we consider the local basis $(h_i\in H,\bar{h}_i\in\bar{H},q_j\in Q)$ introduced in Proposition \ref{proppilocal} and the corresponding dual cobasis $(\kappa^i\in H^*,\bar{\kappa}^i\in\bar{H},\xi^j\in Q^*)$; $(\kappa^i\in H^*,\bar{\kappa}^i\in\bar{H})$ is a basis of $ann\,Q$ and $(\xi^j)$ is a basis of $ann\,P$. We will say that $w\in\chi^k(M)$ is of triple grade $(a,b,c)$ if its local expression is
$$ w=\frac{1}{a!b!c!}w^{l_1...l_ai_1...i_bj_1...j_c}q_{l_1}
\wedge...\wedge q_{l_a}\wedge h_{i_1}\wedge...\wedge h_{i_b}
\wedge\bar{h}_{j_1}\wedge...\wedge\bar{h}_{j_c},$$
where the coefficients are skew-symmetric in each of the three groups of indices and $a+b+c=k$.
The space of such multivectors will be denoted by $\chi^{abc}(M)$.
\begin{prop}\label{sigma2desc} For a quasi-classical generalized CRF manifold, one has a decomposition
\begin{equation}\label{descsigma2} \sigma''=\sigma''_H+\sigma''_{\bar{H}}, \end{equation}
where $\sigma''_H:\chi^{abc}(M)\rightarrow\chi^{a,b+1,c}(M),\;
\sigma''_{\bar{H}}:\chi^{abc}(M)\rightarrow\chi^{a,b,c+1}(M)$ and $\sigma''^{2}_H= \sigma''^{2}_{\bar{H}}=0$, $\sigma''_H\circ\sigma''_{\bar{H}}+\sigma''_{\bar{H}}\circ\sigma''_H=0$.
\end{prop}
\begin{proof}
We shall use the second formula (\ref{exprsigma}) in order to compute $\sigma''$, which has type $(0,1)$, for corresponding arguments as follows:
$$\begin{array}{l}
(\sigma''w)(\xi^{l_1},...,\xi^{l_a},\kappa^{i_1},...,\kappa^{i_{b+e}},\bar{\kappa}^{j_1},
...,\bar{\kappa}^{j_{c-e+1}}),\;\;e=1,...,c+1,\vspace*{2mm}\\
(\sigma''w)(\xi^{l_1},...,\xi^{l_a},\kappa^{i_1},...,\kappa^{i_{b-f+1}},\bar{\kappa}^{j_0},
...,\bar{\kappa}^{j_{c+f}}),\;\;f=1,...,b+1.\end{array}$$

Formula (\ref{exprsigma}) straightforwardly shows that the resulting value is zero if either $e>2$ or $f>2$. We shall prove the same for $e=2,f=2$; the two cases are similar and we give the details for $e=2$ only. Then, the first term of the right hand side of the expression (\ref{exprsigma}) of $\sigma''$ vanishes since the involved value of $w$ does not have the right number of arguments. The second term vanishes because
$\{\xi^l,\kappa^i\}_\pi,\{\xi^l,\bar{\kappa}^i\}_\pi\in ann\,P$	 imply the addition of one more argument in $ann\,P$ and, again, the number of arguments is not the one that yields non-zero values of $w$.
For the same reason, we have to replace $\{\beta_i,\beta_j\}_\pi$ by $pr_{ann\,Q}\{\beta_i,\beta_j\}_\pi$ in the last term of the same expression of $\sigma''$.

In our case, $ann\,Q=H^*+\bar{H}^*$ and we shall get (\ref{descsigma2}) and all the required conclusions if we show that $pr_{H^*}pr_{ann\,Q}\{\kappa_i,\kappa_j\}_\pi=0$. This condition is easily seen to be equivalent to the vanishing of $\{\kappa_i,\kappa_j\}_\pi(U)$ $\forall U\in\bar{H}$. We have
$$\begin{array}{lcl}
\{\kappa_i,\kappa_j\}_\pi(U)&=&<L_{\sharp_\pi\kappa_i}\kappa_j,U> -<L_{\sharp_\pi\kappa_j}\kappa_i,U>-U(\pi(\kappa_i,\kappa_j)\vspace*{2mm}\\	 &=&-<\kappa_j,[\sharp_\pi\kappa_i,U]>+<\kappa_j,[\sharp_\pi\kappa_i,U]>.\end{array}$$
Since a straightforward calculation leads to the relation
$$<L_{\sharp_\pi\kappa_i}\kappa_j,U>+<L_U\kappa_j,\sharp_\pi\kappa_i>=U(\pi(\kappa_i,\kappa_j)),$$
we get
$$\begin{array}{lcl}
\{\kappa_i,\kappa_j\}_\pi(U)&=&-<L_U\kappa_j,\sharp_\pi\kappa_i> -<L_U\kappa_i,\sharp_\pi\kappa_j>-U(\pi(\kappa_i,\kappa_j))
\vspace*{2mm}\\&=&(L_U\pi)(\kappa_i,\kappa_j)=0,
\end{array}$$
because of the second integrability condition (\ref{bunCRFPs}).
\end{proof}

\vspace*{5mm}
{\small Department of Mathematics, University of Haifa, Israel, vaisman@math.haifa.ac.il}

\begin{thebibliography}{xx}
\bibitem{Bl} D. E. Blair, Riemannian geometry of contact and symplectic manifolds, Progress in Math., vol. 203, Birkh\"auser, Boston, 2002 (second edition 2010).
\bibitem{CZ} A. S. Cattaneo and M. Zambon, Coisotropic embeddings in Poisson manifolds, Trans. Amer.
Math. Soc. 361(7)  (2009), 3721–3746.
\bibitem{Cr} M. Crainic, Generalized complex structures and Lie brackets, Bulletin of the Brazilian Mathematical Society, 42 (2011), 559-578.
\bibitem{DT} S. Dragomir and G. Tomassini, Differential Geometry and
Analysis on CR Manifolds, Progress in Math., Vol. 246,
Birkh\"auser Verlag, Basel, 2006.
\bibitem{DK} T. Duchamp anf M. Kalka, J. Diff. Geom., 14 (1979), 317-337.
Deformation theory for holomorphic foliations
\bibitem{GD} I. M. Gelfand and I. Ya. Dorfman, The Schouten bracket and Hamiltonian operators, Funkt. Anal. Prilozhen. 14 (3) (1980), 71-74.
\bibitem{Gengoux} C. Laurent-Gengoux, M. Sti\'enon and P. Xu, Holomorphic Poisson Manifolds and Holomorphic Lie Algebroids, Int. Math. Res. Not. IMRN 2008, Art. ID rnn 088, 46 pp.
\bibitem{Galt} M. Gualtieri, Generalized complex geometry, Ph.D.
thesis, Univ. Oxford, 2003; arXiv:math.DG/0401221.
\bibitem{H} N. Hitchin, Generalized Calabi-Yau manifolds, Q. J. Math., 54 (2003),  281–308.
\bibitem{Mor} A. Morimoto, On normal almost contact structures, J. Math. Soc. Japan, 15 (1963), 420-436.
\bibitem{V73} I. Vaisman, Cohomology and Differential Forms, Marcel Dekker, Inc., New York, 1973.
\bibitem{VP-Nij} I. Vaisman, The Poisson-Nijenhuis manifolds revisited, Rend. Sem. Mat. Torino, 52 (1994), 377-394.
\bibitem{Psgeom} I. Vaisman, Lectures on the geometry of Poisson manifolds, Progress in Math. vol. 118, Birkh\"auser, Boston, 1994
\bibitem{V-Dircoupl} I. Vaisman, Foliation-coupling Dirac structures, J. Geom. Physics, 56 (2006), 917-938.
\bibitem{V-gcm} I. Vaisman, Reduction and submanifolds of generalized complex
manifolds, Diff. Geom. Appl., 25 (2007), 147-166.
\bibitem{Vstable} I. Vaisman, Dirac structures and generalized complex structures on $TM\times\mathds{R}^h$, Adv. Geom., 7 (2007), 453-474, DOI 10/1515/ADVGEOM.2007.29.
\bibitem{VCRF} I. Vaisman, Generalized CRF-structures,
Geometriae Dedicata, 133 (2008), 129-154, dx.doi.org/10.1007/s10711-008-9239-z
\bibitem{Y} K. Yano, On a structure defined by a tensor field $f$ of type $(1,1)$ satisfying $f^3+f=0$, Tensor, 14 (1963), 99-109.
\end{thebibliography}
\end{document}